\documentclass[a4paper,Haag duality]{mathscan}
\newtheorem{thm}{Theorem}[section]

\newtheorem{lem}[thm]{Lemma}        
\theoremstyle{definition}           
\newtheorem{rem}[thm]{Remark}       

\newcommand{\NI}{\noindent}

\newcommand{\bea}{\begin{eqnarray}}
\newcommand{\eea}{\end{eqnarray}}

\def \b #1 {\bf #1}
\newcommand{\IR}{I\!\!R}

\newcommand{\IC}{\mathbb{C}}

\newcommand{\cal}{\mathcal}

\newcommand{\clm}{{\cal M}}

\newcommand{\clf}{{\cal F}}

\newcommand{\clh}{{\cal H}}
\newcommand{\clp}{{\cal P}}

\newcommand{\clb}{{\cal B}}

\newcommand{\clj}{{\cal J}}
\newcommand{\cln}{{\cal N}}

\newcommand{\raro}{\rightarrow}

\newcommand{\vsp}{\vskip 1em}

\def \qed {\hfill \vrule height6pt width 6pt depth 0pt}
\newcommand{\be}{\begin{equation}}
\newcommand{\ee}{\end{equation}}
\newcommand{\ben}{\begin{eqnarray*}}
\newcommand{\een}{\end{eqnarray*}}
\begin{document}

\title {A mean ergodic theorem in von-Neumann algebras }

\author{ Anilesh Mohari }
\thanks{...}

\address{ The Institute of Mathematical Sciences, Taramani, CIT Campus, Chennai-600113}
\email{anilesh@imsc.res.in}

\begin{abstract} 
We explore a duality between von-Neumann's mean ergodic theorem in von-Neumann algebra and Birkhoff's mean ergodic theorem in the pre-dual Banach space of von-Neumann algebras.  
Besides improving known mean ergodic theorems on von-Neumann algebras, we prove Birkhoff's mean ergodic theorem for any locally compact second countable amenable group action 
on the pre-dual Banach space.

\end{abstract}
\maketitle 

\section{ Introduction:}

\vsp 
Let $\clh$ be a complex Hilbert space with inner product $<.,.>$ assumed to be conjugate linear in the first variable and $\clb(\clh)$ be the set of 
bounded operators on $\clh$. A $*$-sub-algebra $\clm$ of $\clb(\clh)$ is called von-Neumann algebra if $\clm$ is closed in weak$^*$ topology on $\clb(\clh)$ 
[Tak]. A triplet $(\clm,\tau,\phi)$ is called quantum dynamical system where $\clm$ is a von-Neumann algebra and $\tau$ is an unital completely positive map 
[St] with a faithful normal invariant state $\phi$. 

\vsp 
Let $(\clm,\tau,\phi)$ be a quantum dynamical system. It is quite some time now that it is known 
[La,Fr] that 
\be 
s_n(x)={1 \over n} \sum_{1 \le k \le n} \tau^n(x) \raro E(x)
\ee 
as $n \raro \infty$ in weak$^*$ topology i.e. $\sigma-$weak operator topology i.e. for each $\psi \in \clm_*$ 
$\psi(s_n(x)) \raro \psi(E(x))$ for all $x \in \clm$ as $n \raro \infty$. Given a sequence of weakly convergent normal states 
$\psi_n$ on $\clm$, the sequence need not be strongly convergent i.e. need not be Cauchy in Banach space norm of $\clm_*$. Such 
a statement however true if $\clm$ is a type-I von-Neumann algebra with center completely atomic. In other situation it is known 
to be false [De]. Thus the general theory no way ensures Birkhoff theorem which says that 
\be 
||\psi \circ s_n-\psi \circ E || \raro 0
\ee
as $n \raro \infty$ for all element $\psi \in \clm_*$. At this point we note one interesting point that (2) is valid for all $\psi \in \clm_*$
if (2) is valid for a dense subset of $\clm_*$ in the Banach space norm. For a proof one can use standard $3 \epsilon$ argument as the 
maps $||s_n-E|| \le 2$ for all $n \ge 1$. A Proof for (2) appears first in E.C. Lance's paper when $\tau$ is an $*$-automorphism. 
Preceding A. Frigerio's work [Fr], B. Kummerer [Ku] developed a general theory which includes a proof for (2) however assuming (1). For 
more detail account we refer to monograph [Kr]. Note that $s_n$ is a sequence of unital completely positive map. Our strategy is to develop a general theory 
valid for a sequence of unital completely positive maps on $\clm$ which we now describe below: 
   
\vsp 
Here our argument is fairly general which uses faithfulness of the normal state $\phi$ to identify $\clm$ with it's standard form [BR] 
$(\clh_{\phi},\pi_{\phi}(\clm),\clp_{\phi},\clj_{\phi},\Delta_{\phi},\zeta_{\phi})$ associated with $\phi$, where $\clh_{\phi}$ is the GNS space 
associated with $\phi$, $\Delta_{\phi},\clj_{\phi}$ are Tomita's modular operators associated with the closable anti-linear map 
$S_0:\pi_{\phi}(x)\zeta_{\phi} \raro \pi_{\phi}(x^*)\zeta_{\phi}$ with polar decomposition 
$\overline{S_0}= \clj_{\phi} \Delta^{1 \over 2}_{\phi}$   
and $\clp_{\phi}$ is Araki's pointed positive cone given by the closer 
of $\{ \pi_{\phi}(x) \clj_{\phi} \pi_{\phi}(x) \clj_{\phi} \zeta_{\phi}: x \in \clm \}$ in $\clh_{\phi}$. Tomita's theorem says that 
modular group $\sigma_t(x)=\Delta^{it}x\Delta^{-it},\;x \in \clb(\clh_{\phi})$ preserves $\pi_{\phi}(\clm)$, ( so it's commutant 
$\pi_{\phi}(\clm)'=\{ y \in \clb(\clh_{\phi}):yx=xy \}$ ) and $\clj_{\phi}\pi_{\phi}(\clm)\clj_{\phi}=\pi_{\phi}(\clm)'$.  In the following 
we will use symbol $x$ for $\pi_{\phi}(x)$ for simplicity of notation. 

\vsp 
Let $\clm_+$ be be the non-negative elements in $\clm$. A map $\tau:\clm \raro \clm$ is called positive if $\tau(\clm_+) \subseteq \clm_+$. We also set  
$$P^{1 \over 2}_{\phi}(\clm)=\{ \tau:\clm \raro \clm, \phi \tau \le \phi, \tau(1) \le 1,\;\;\mbox{positive map and}$$
$$\;\;\phi(\tau(x^*)\tau(x)) \le \phi(x^*x),\;x \in \clm \} $$
For an element $\tau \in P^{1 \over 2}_{\phi}(\clm)$, we define a contractive operator $T$
defined by 
\be 
T x \zeta_{\phi} = \tau(x)\zeta_{\phi}
\ee 
for all $x \in \clm$. It is clear that $P^{1 \over 2}_{\phi}(\clm)$ is a convex set as convex 
combination of two contractive operator on $\clh_{\phi}$ is also contractive. A natural topology on 
this convex is the relative topology i.e. we say a net $\tau_{\alpha} \raro \tau$ converges in 
$P^{1 \over 2}_{\phi}(\clm)$ if $T_{\alpha} \raro T$ in strong operator topology.   

\vsp 
We also have a unique dual element 
$\tau' \in P^{1 \over 2}_{\phi}(\clm')$ satisfying 
\be 
<y'\zeta_{\phi},\tau(x)\zeta_{\phi}>=<\tau'(y')\zeta_{\phi},x\zeta_{\phi}>
\ee 
for all $x \in \clm, y' \in \clm'$. We also denote by $T'$ the contractive operator on $\clh_{\phi}$ 
defined by 
$$T'y'\zeta_{\phi} = \tau'(y')\zeta_{\phi}$$  
Note by adjoin relation (4) above $T'$ is the Hilbert space adjoint map of $T$ and thus a contraction as
well. Thus $\tau' \in P^{1 \over 2}_{\phi}(\clm')$ if and only if $\tau \in P^{1 \over 2}_{\phi}(\clm)$.

\vsp 
Our main mathematical result says the following: 

\vsp 
\begin{thm} 
Let $(\clh_{\phi},\pi_{\phi}(\clm),\clj_{\phi},\Delta_{\phi},\clp_{\phi},\zeta_{\phi})$ be the 
standard form associated with a faithful normal state $\phi$ on von-Neumann algebra $\clm$. We identify $\clm$ with 
$\pi_{\phi}(\clm)$ in the following. Let $(\tau_n:n \ge 1 )$ be a sequence of elements in $P^{1 \over 2}_{\phi}(\clm)$ 
and $T'_n$ be the sequence of contractive operator associated with the dual sequence of elements $(\tau'_n:n \ge 1)$ in $P^{1 \over 2}_{\phi}(\clm')$. 
Then 
\be 
||\psi \tau_n -\psi \tau_m|| \raro 0
\ee 
as $m,n \raro \infty$ in $\clm_*$ for any $\psi \in \clm_*$ if 
$||[T'_n-T'_m]f|| \raro 0$ as $m,n \raro \infty$ for all $f \in \clh_{\phi}$. Furthermore if $||[T_n-T_m]f|| \raro 0$ as $m,n \raro \infty$ for all 
$f \in \clh_{\phi}$ then there exists an element $\tau \in P^{1 \over 2}_{\phi}(\clm)$ such that 
\be 
||\psi \tau_n -\psi \tau|| \raro 0
\ee 
as $n \raro \infty$ in $\clm_*$ for any $\psi \in \clm_*$. Conversely the first part of the statement is also true if the modular group 
commutes with each $\tau_n$ on $\clm$. 
\end{thm}  
 
\vsp 
Duality argument used here is not entirely new and can be traced back to E.C. Lance work (Theorem 3.1 (iv) in [La]) where 
$\tau$ is an automorphism. However Lance's analysis uses dual of an automorphism is also an automorphism and thus made it possible to  use 
Kaplansky's density theorem to conclude (2). Thus the real merit of our method here is it's simplicity and 
it's generality. In general the converse statement of Theorem 1.1 without commuting property with modular group is false [Mo3]. For individual 
ergodic theorem in non commutative frame work of von-Neumann algebras, we refer to [La,Ku,JXu].  
  
\vsp 
The paper is organized as follows: In our next section we prove our main result Theorem 1.1. As an application of Theorem 1.1, in section 3 we prove Birkhoff's mean ergodic theorem 
for a second countable locally compact amenable group action $G$ on a von-Neumann algebra $\clm$.    

\vsp 
I thank Rainer Nagel for drawing my attention to the paper [Ku] and Q. Xu for drawing my attention to the paper [La]. I also thank anonymous referee for helpful comments which made it 
possible to re-organize the paper to it's present form.

\section{Mean ergodic theorem for completely positive map:}

\vsp 
A positive map $\tau:\clm \raro \clm$ is called normal if $l.u.b.\tau(x_{\alpha}) = \tau(l.u.b.x_{\alpha})$ for any increasing net $x_{\alpha} \in \clm_+$ bounded from above 
where $l.u.b.$ is the least upper bound. Let $\phi$ be a faithful normal $\tau$-invariant (sub-invariant) state on $\clm$ i.e. $\phi(\tau(x))=\phi(x)$ $(\; \phi(\tau(x)) \le \phi(x) \;)$ 
for all $x \in \clm_+$. The result that follows in this section has a ready generalization to a faithful weight [Tak] which we avoid as it makes very little qualitative improvement of our 
results once we incorporate Radon-Nykodym theorem for weights. One interesting observation here that once we know a positive map $\tau$ admits a faithful normal invariant or sub-invariant state, 
then $\tau$ is automatically normal [AC]. We start with a non-commutative counter part of E. Nelson's theorem [Ne] on commutative von-Neumann algebra.  

\vsp 
\begin{lem}
Let $\phi$ be a faithful normal state on $\clm$ which we identify with standard representation $\pi_{\phi}(\clm)$ where $(\clh_{\phi},\pi_{\phi},\zeta_{\phi})$ 
is the GNS space associated with $\phi$, $\phi(x)=<\zeta_{\phi},\pi_{\phi}(x)\zeta_{\phi}>$ and $\zeta_{\phi}$ is a cyclic and separating unit vector for $\pi_{\phi}(\clm)$ 
in $\clh_{\phi}$. Let $\tau \in P^{1 \over 2}_{\phi}(\clm)$. Then there exists a positive map $\tau':\pi_{\phi}(\clm)' \raro \pi_{\phi}(\clm)'$ so that 
\be 
<\tau'(y')\zeta_{\phi},x\zeta_{\phi}>=<y'\zeta_{\phi},\tau(x)\zeta_{\phi}>
\ee 
for all $x \in \pi_{\phi}(\clm)$ and $y' \in \pi_{\phi}(\clm)'$ and $\tau' \in P^{1 \over 2}_{\phi}(\clm')$. Further for two such positive map $\tau_1$ and $\tau_2$ we have  
\be 
|\psi_{y'_1,y'_2}(\tau_1(x)-\tau_2(x))| \le ||(\tau_1'-\tau_2')((y'_2)^*y'_1)\zeta_{\phi}|| ||x\zeta_{\phi}|| \le ||(\tau'_1-\tau'_2)((y'_2)^*y'_1)\zeta_{\phi}||\;||x||
\ee   
where $\psi_{y'_1,y'_2}(x)=<y'_1\zeta_{\phi},xy'_2\zeta_{\phi}>$. 

\end{lem} 

\begin{proof} 
First part of the statement is not new and details have been worked out in Chapter 8 in [OP] which followed closely [AC]. Existence part is a standard result proved by a simple application of Dixmier's fundamental lemma [Di] which a normal state $\psi$ satisfying $\psi \le \lambda \phi$ on $\clm_+$ for some $\lambda > 0$ admits a representation 
$\psi(x) =< y'\zeta_{\phi},x\zeta_{\phi}>$ where $y'$ is non-negative element in the commutant $\pi_{\phi}(\clm)'$ of $\pi_{\phi}(\clm)$ and such an element 
is unique. For any non-negative element $y' \in \pi_{\phi}(\clm)'$ we consider non-negative normal functional $\psi_{y'}(x)=<y'\zeta_{\phi},\tau(x)\zeta_{\phi}>$ 
and check that $\psi_{y'} \le ||y'|| \phi $ and thus by Dixmier's lemma we get an unique non-negative element $\tau'(y') \in \pi_{\phi}(\clm)'$ so that (7) is satisfied. 
We can extend now the definition of $\tau'(y')$ for an arbitrary self adjoint element by linearity using the decomposition $y'=y'_+-y'_-$ with $y'_+y'_-=0$. For an arbitrary element $y$ we use $\pi_{\phi}(\clm)'=\pi_{\phi}(\clm)'_h+i\pi_{\phi}(\clm)'_h$, $\pi_{\phi}(\clm)'$ being a $*$-algebra. It is simple to check that (7) hold for any $y' \in \pi_{\phi}(\clm)', x  \in \pi_{\phi}(\clm).$  

\vsp 
That $\tau'$ takes non-negative element of $\clm'$ to non-negative elements follows by our construction and 
also note that $\tau'(I) \le I$ since $\phi(\tau(x)) \le \phi(x)$ for all $x \in \clm_+$ i.e. $\phi(\tau'(I)x) \le \phi(x)$ for some $\tau'(I) \in \clm'$. Since $\zeta_{\phi}$ 
is cyclic for $\clm$ and we have $<x\zeta_{\phi}, (I - \tau'(I))x\zeta_{\phi}> \ge 0$ for all $x \in \clm$, we conclude $\tau'(I) \le I$. Similarly $\phi (\tau'(x)) \le \phi(x)$ 
for $x \in \clm_+$ since $\tau(I) \le I$. The vector state $y' \raro <\zeta_{\phi},y'\zeta_{\phi}>$ being normal faithful on $\clm'$, invariant or more generally sub-invariant
property for $\tau'$ will ensure normal property of $\tau'$ from normality of the vector state.

\vsp 
The inequality (8) is fairly obvious by Cauchy-Schwarz inequality. 

\end{proof} 

\vsp 
\begin{proof}(Theorem 1.1) We use the inequality (8) given in Lemma 2.1 and our hypothesis on $\tau_n$ to prove that 
$|||\psi_{y'_1,y'_2}(\tau_n-\tau_m)|| \le ||(\tau_n'-\tau_m')((y'_2)^*y'_1)\zeta_{\phi}|| \raro 0$ as $m,n \raro \infty$ for 
all $y'_1,y'_2 \in \clm'$. Since the set of normal states $\{\psi_{y'_1,y'_2}:y'_1,y'_2 \in \clm' \}$ are total in $\clm_*$, 
by standard $3 \epsilon$ argument, we get (5).

\vsp 
For the second part of the statement we follow here [Fr] and apply compactness property for the bounded net $\tau_n(x)$ as $||\tau_n(x)|| \le ||x||$ to claim that 
it has unique limit point. Let $x^1_{\infty},x^2_{\infty}$ be two such limit points. If so then  
$x^1_{\infty}\zeta_{\phi}=x^2_{\infty}\zeta_{\phi}$ as $\tau_n(x)\zeta_{\phi} \raro Tx\zeta_{\phi}$ 
in strong operator topology, where $T$ is strong operator limit of $T_n$. Now by separating property of $\zeta_{\phi}$ for $\clm$, we conclude that $x^1_{\infty}=x^2_{\infty}$. 
Limit point being unique we conclude that $\tau_n(x) \raro x_{\infty}$ for
some $x_{\infty}$ in weak$^*$ topology. It is now a routine work to check that the map 
$x \raro \tau(x)=x_{\infty}$ is a positive map with $\phi \circ \tau \le \phi$ on $\clm_+$ and $\tau(I) \le I$. Thus $\tau$ is also normal. We claim that  
$\tau \in P^{1 \over 2}_{\phi}$. We compute the following simple steps with $x \in \clm$ and $y' \in \clm'$:
$$|<y'\zeta_{\phi},\tau(x)\zeta_{\phi}>|$$
$$=\mbox{lim}_{n \raro \infty}|<y'\zeta_{\phi},\tau_n(x)\zeta_{\phi}>|$$
$$\le ||y'\zeta_{\phi}||| \mbox{sup}_{n \ge 1} ||\tau_n(x)\zeta_{\phi}||$$
$$\le ||y'\zeta_{\phi}||\;||x\zeta_{\phi}||$$ 
This shows $\tau \in P^{1 \over 2}_{\phi}$. That $\tau$ satisfies (6) follows 
by (8) now as $T'_n \raro T'$ in strong operator topology.    

\vsp 
We need to prove the converse part of the first part of the statement which is not so immediate. We write the following identity: 
$$||\Delta^{1 \over 4}\tau(x)\zeta_{\phi}||^2 $$
$$= <\tau(x)\zeta_{\phi},\Delta^{1 \over 2}\tau(x)\zeta_{\phi}> $$
$$= <\tau(x)\zeta_{\phi},\clj \tau(x^*) \clj \zeta_{\phi}> $$
$$= <x\zeta_{\phi}, \tau'(\clj \tau(x^*) \clj) \zeta_{\phi}>$$
$$=\overline{<\clj x \clj \zeta_{\phi},\tilde{\tau}\tau(x^*)) \zeta_{\phi}>}$$
$$=<\clj x \clj \zeta_{\phi}, \tilde{\tau} \tau(x^*) \zeta_{\phi}>$$ 
where $\tilde{\tau}(x)=\clj\tau'(\clj x \clj) \clj$ for $x \in \clm$.
Similarly 
$$||\Delta^{1 \over 4}(\tau_{\alpha}-\tau)(x) \zeta_{\phi}||^2  $$
$$= <\clj x \clj \zeta_{\phi}, (\tilde{\tau}_{\alpha}-\tilde{\tau}) \circ (\tau_{\alpha}-\tau) (x^*) \zeta_{\phi}>$$ 
$$\le ||\psi_x \circ (\tilde{\tau}_{\alpha}-\tilde{\tau}) \circ (\tau_{\alpha}-\tau) (x^*)|| ||x\zeta_{\phi}||$$  
$$\le 2 ||x||^2 ||\psi_x \circ (\tilde{\tau}_{\alpha}-\tilde{\tau})||$$  
where $\psi_x(y)=<\clj x \clj \zeta_{\phi},y \zeta_{\phi}>$, a normal functional on $\clm$. This shows if $||\psi \circ (\tilde{\tau}_{\alpha} - \tilde{\tau})|| \raro 0$ for
all normal functional, then $\Delta^{1 \over 4}\tau_{\alpha}(x)\zeta_{\phi} \raro \Delta^{1 \over 4}\tau(x)\zeta_{\phi} $. So far we did not use our hypothesis
that each $\tau_{\alpha}$ commutes with modular group $\sigma_t$ on $\clm$. By restricting to analytic elements $x$ of the form $x_{\delta}= \int_{\IR}\sigma_t(x)d\mu_{\delta}$, where 
$\mu_{\delta}$ is the Gaussian probability measure on $\IR$ with variance $\delta > 0$, we check using commuting property with each 
$\tau_{\alpha}$ that $\tau_{\alpha}(x_{\delta})$ is also an analytic element for $\sigma_t$ and thus $\Delta^{1 \over 4} \tau_{\alpha}(x_{\delta})\zeta_{\phi} = 
\tau_{\alpha}(\sigma_{-{i \over 4}}(x_{\delta}))\zeta_{\phi}$ same is also true for $\tau$ as it also commutes with modular group $(\sigma_t)$ being weak$^*$ limit 
of commuting elements $\tau_{\alpha}$. Thus $||T_{\alpha}f -Tf|| \raro 0$ for $f \in \clh_0=\{x_{\delta}\zeta_{\phi}: x \in \clm, \delta > 0 \}$. 

\vsp 
Since $x_{\delta} \raro x$ in weak$^*$ topology as $\delta \raro 0$ [BR,Proposition 2.5.22], we get $\clh_0$ is dense ( for any $g \in \clh_0^{\perp}$ we have $<f,x_{\delta}\zeta_{\phi}>=0$ and since $x_{\delta} \raro x$ in 
weak operator topology, we get $<f,x\zeta_{\phi}>=0$. But $\zeta_{\phi}$ is cyclic for $\clm$ and so $g=0$. ). Thus $T_{\alpha}f \raro Tf$ strongly for all $f$ on a dense subspace of $\clh$ 
and so does for all $f \in \clh$ since the family of bounded operators involved in the limit are uniformly bounded, all are being contraction.
\end{proof}        

\vsp 
\begin{lem}  
Let $(\clm,\tau,\phi)$ be as in Lemma 2.1 i.e. $\tau \in P^{1 \over 2}_{\phi}(\clm)$. The map $T:x\zeta_{\phi} \raro \tau(x)\zeta_{\phi}$ has 
a unique contractive extension on $\clh_{\phi}$
and 
\be 
s_n(T)={1 \over n} \sum_{0 \le k \le n-1} T^k \raro P
\ee 
in strong operator topology where $P$ is the projection on the closed subspace $\{f: Tf=f \}$.
\end{lem} 
\vsp 
\begin{proof} 
Since $\tau \in P^{1 \over 2}$, $T$ is a contraction on $\clh_{\phi}$ and the relation (9) is von-Neumann mean ergodic theorem [Ha] for contraction. 
\end{proof} 

\vsp         
\begin{thm} 
Let $\tau \in P^{1 \over 2}_{\phi}(\clm)$. Then 
$$s_n = {1 \over n} \sum_{0 \le k \le n-1} \tau^k \in CP^{1 \over 2}_{\phi}(\clm)$$ 
and 
\be 
||\psi \circ s_n - \psi \circ E|| \raro 0
\ee 
as $n \raro \infty$ for any $\psi \in \clm_*$ where $E:\clm \raro  \cln$ is a positive map with range equal to $\cln=\{x \in \clm:\tau(x)=x \}$ and $E \in P^{1 \over 2}_{\phi}(\clm)$. Further 
\vsp  
\NI (a) $Px\zeta_{\phi}=E(x)\zeta_{\phi}$ and $P$ is equal to the closed space generated by $\{x\zeta_{\phi}:x \in \cln \}$. 
\vsp   
\NI (b) $\cln$ is a von-Neumann algebra if and only if $E$ is unital ( so $\tau$ is unital ) satisfying bi-module property  
\be 
yE(x)z=E(yxz)
\ee 
for all $y,z \in \cln$ and $x \in \clm$. 
\end{thm} 

\vsp 
\begin{proof}
By symmetry of the argument used we also get $s_n(T') \raro P$ in strong operator topology, where $T'$ is the Hilbert space adjoint map of $T$. Since $s_n(T) \raro P$ in strong operator 
topology and $s_n(T') \raro P$ in strong operator topology, there exists $E \in P^{1 \over 2}_{\phi}(\clm)$ satisfying $||\psi \circ s_n - \psi \circ E|| \raro 0$ in $\clm_*$ for all 
$\psi \in \clm_*$ by Theorem 1.1. 

\vsp 
Since $ns_n(\tau(x))= n\tau (s_n(x))=(n+1)s_{n+1}(x)-x$, we get $E(\tau(x))=\tau(E(x))=E(x)$ for all $x \in \clm$. This shows $E(x) \in \cln$ and $E^2=E$. That 
$Px\zeta_{\phi}=\mbox{w-lim}_{n \raro \infty}s_n(T)x\zeta_{\phi}=\mbox{w-lim}s_n(x)\zeta_{\phi}=E(x)\zeta_{\phi}$. Now we fix a vector $f$ so that $Pf=f$ and $<f,x\zeta_{\phi}>=0$ for
all $x \in \cln$. Then for any $x \in \clm$ we have $<f,x\zeta_{\phi}>=<Pf,x\zeta_{\phi}>=<f,Px\zeta_{\phi}>=<f,E(x)\zeta_{\phi}>=0$. Thus $f=0$ since $\zeta_{\phi}$ is cyclic for 
$\clm$. This completes the proof of (a).  

\vsp 
We are left to prove (b). Bi-module property of $E$ follows from a general result [CE] provided $\cln$ is also a von-Neumann sub-algebra. Here we give a direct proof as follows. By our 
assumption that $\cln$ is a von-Neumann algebra we have $E(I)=I$. For bi-module property, it is enough if we show $E(x)z=E(x)z$ for all $x \in \clm$ and $z \in \cln$. By separating property 
of $\zeta_{\phi}$ for $\cln$, it is enough to verify following equalities:
$$<y^*\zeta_{\phi},E(zx)\zeta_{\phi}>=<y^*\zeta_{\phi},P zx \zeta_{\phi}>$$
$$=<Py^*\zeta_{\phi},zx\zeta_{\phi}>=<z^*y^*\zeta_{\phi},x \zeta_{\phi}>$$
(since $z^*y^* \in \cln$ being an algebra)
$$=<Pz^*y^*\zeta_{\phi},x \zeta_{\phi}>=<z^*y^*\zeta_{\phi},Px\zeta_{\phi}>$$
$$=<y^*\zeta_{\phi},zE(x)\zeta_{\phi}>$$  
This shows $zE(x)=E(zx)$ for all $z \in \cln$ and $x \in \clm$. We get the requited property 
by taking conjugate as $\cln,\clm$ are $*$-closed. 

\vsp 
Conversely bi-module property shows that $\cln$ is a $*$-algebra. Since $\tau$ is a normal map, $\cln$ is either $0$ or a non-degenerate von-Neumann algebra. Since by our assumption 
for the converse state $E$ is unital $I \in \cln$ and $\cln$ is a von-Neumann algebra. We are left to show $E(I)=I$ holds if and only if $\tau(I)=I$. Since $0 \le \tau(I) \le I$ we h
ave $0 \le \tau^{k+1}(I) \le \tau^k(I) \le I$ and thus $0 \le E(I)=\mbox{w-lim}\tau^k(I) \le \tau(I) \le I$. Thus $E(I)=I$ holds if and only if $\tau(I)=I$.      
\end{proof}  

\vsp 
We say a positive map $\tau$ on $\clm$ is an element in $P^{KS}(\clm)$ if 
\be 
\tau(x^*)\tau(x) \le \tau(x^*x)
\ee
for all $x \in \clm$. For such a $\tau$, we have $\tau(1) \in \clm_+$ and $\tau(1)\tau(1) \le \tau(1)$ 
i.e. $\tau(1) \le 1$. Further for some $x \in \clm$ we have   
$$ 
\tau(x^*)\tau(x) = \tau(x^*x) 
$$ 
if and only if 
\be 
\tau(x^*)\tau(y)=\tau(x^*y)
\ee
for all $y \in \clm$. This shows that $\clf=\{x: \tau(x^*)\tau(x)=\tau(x^*x),\;\tau(x)\tau(x^*)=\tau(xx^*) \}$
is $C^*$-sub algebra of $\clm$ and normal property of $\tau$ ensures that $\clf$ is a von-Neumann algebra if $I \in \clf$ 
i.e. $\tau(I)$ is a projection.  

\vsp 
We denote by 
$$P^{KS}_{\phi}(\clm)=\{ \tau \in P^{KS}(\clm): \phi(\tau (x)) \le \phi(x),\;x \in \clm_+ \}$$ 
Kadison-Schwarz inequality (12) and sub-invariant property $\phi \tau \le \phi$ on $\clm_+$ says 
that $P^{KS}_{\phi}(\clm) \subseteq P^{1 \over 2}_{\phi}(\clm)$. Whether $\tau \in P^{KS}_{\phi}(\clm)$ 
if and only if $\tau' \in P^{KS}_{\phi}(\clm')$ remains unknown. 

\vsp 
For $\tau \in P^{KS}_{\phi}(\clm)$, we claim that $\cln=\{x \in \clm: \tau(x) = x \}$ is a $*$-sub-algebra of $\clm$.  
Since $x^*x=\tau(x^*)\tau(x) \le \tau(x^*x)$ for all $x \in \cln$ and so $\phi(x^*x) \le \phi(\tau(x^*x)) \le \phi(x^*x)$ i.e.
by faithful property of $\phi$ we have $\tau(x^*x)=x^*x$. This shows that $\tau(x^*y)=\tau(x^*)\tau(y)=x^*y$ for
all $x,y \in \cln$. Further by (13) we have bi-module property: 
\be 
\tau(yxz)=y\tau(x)z
\ee 
for all $y,z \in \cln$ and $x \in \clm$. Thus $\cln$ is a von-Neumann algebra if and only if $\tau$ is unital.     

\begin{rem} 
One can deduce bi-module property of $E$ defined in Theorem 2.3 from bi-module property of $\tau$ as follows. 
Applying bi-module property of $\tau$ repeatedly, we get
$$y\tau^k(x)z = \tau^k(yxz)$$
for all $y,z \in \cln$ and $x \in \clm$. Thus same bi-module property holds for $s_n$ and so for it's weak$^*$ limiting map $E$. 
\end{rem}

\vsp 
A positive map $\tau:\clm \raro \clm$ is called $n$-{\it positive} if 
$$\tau \otimes I_n:(x^i_j) \raro (\tau(x^i_j))$$ 
is also positive on $M_n(\clm)$, $(n \times n)$ matrices with elements in $\clm$. $\tau$ is called completely positive [St] if $\tau \otimes I_n$ is positive for each $n \ge 1$. 

\begin{lem} 
Let $(\clm,\tau,\phi)$ be as in Lemma 2.1. Then $\tau'$ is $n-$positive if and only if $\tau$ is $n-$positive.
\end{lem} 
 
\begin{proof}  
This is well known and for a proof using Tomita's modular theory we refer to chapter 8 in [OP]. We can avoid use of Tomita's modular 
relation while proving $n$-positivity of $\tau'$ as follows. In the following without loss of generality we assume $\clm$ is identified with it's standard 
form associated with $\phi$ as in Lemma 2.1. We have already proved that $\tau'$ is positive in Lemma 2.1. For $n$ positivity of $\tau'$, 
we consider the GNS space $(\clh_n,\pi_n,\zeta_n)$ of the faithful normal invariant state $\phi_n= \phi \otimes tr_n$ on $M_n(\clm)$, $n \times n$ 
matrices with entries in $\clm$. Now we consider the positive map $\hat{\tau}_n: \pi_n(M_n(\clm)) \raro \pi_n(M_n(\clm))$ defined by 
$\hat{\tau}_n(\pi_n((x^i_j)))= (\pi_n(\tau(x^i_j)))$ assuming $\tau$ is $n-$positive. We have $\clh_n \equiv \clh_{\phi} \otimes \clh_{tr_n}$ where
$(\clh_{tr_n},\rho_n,\Omega_n)$ is the GNS representation of $M_n(\IC)$ with respect to normalized trace $tr_n$. If we identify $\clh_{tr_n} \equiv \IC^n \otimes \IC^n$ 
then $\rho_n(\!M_n(\IC)) \equiv \!M_n(\IC) \otimes I_n$ and $\pi_n(\!M_n(\IC)) \equiv \clm \otimes \!M_n(\IC) \otimes I_n$ and $\pi_n(\!M_n(\IC))' \equiv \clm' \otimes I_n \otimes \!M_n(\IC)$.   

\vsp 
Thus $\hat{\tau}_n \equiv \tau \otimes \rho_n$.  With this identification it is clear now that $(\hat{\tau}_n)' \equiv \tau' \otimes \rho'_n$ where $\rho'_n$ is adjoint map of $\rho_n$ given by Lemma 2.1 
and positive. Further $\tau' \otimes \rho'_n \equiv \tau' \otimes I_n$ via the isomorphism $\pi_n(M_n(\clm))' \equiv M_n(\clm')$. Thus positive property of $\tau'_n=\tau' \otimes I_n$ follows 
from that of $\tau_n=\tau \otimes I_n$.  
\end{proof}

\begin{rem}
It is a well known that for 2-positive map $\tau$ with $\tau(I) \le I$, Kadison-Schwarz inequality (12) holds [Ka]. Thus bi-module property (14) holds for 2-positive 
unital $\tau$. In such a case $E$ is also unital $2-$positive map satisfying bi-module property (11). As an application of Theorem 2.3 we get a proof for Lance's 
theorem for unital 2-positive map $\tau$.       
\end{rem}
             
\vsp 
\begin{thm}  
Let $(\clm,\tau_t,t \ge 0,\phi)$ be a semi-group of completely positive unital map with a faithful normal 
invariant state on $\clm$. Then there exists a norm one projection $E:\clm \raro  \cln$ onto $\cln$ so that 
\be 
||\psi \circ s_{\lambda} - \psi \circ E|| \raro 0
\ee
as $\lambda \raro 0$ for any $\psi \in \clm_*$ where 
$$s_{\lambda} = \lambda \int e^{- \lambda t} \tau_t dt$$ 
and $E$ is the norm one normal projection onto $\cln=\{x: \tau_t(x)=x,\;t \ge 0 \}$.    
\end{thm} 

\vsp 
\begin{proof} 
Proof goes along the same line that of discrete time dynamics [Fr]. We indicate here a proof leaving the details to reader. We set contractive operator 
$$s_{\lambda}(T) = \lambda \int e^{- \lambda t} T_t dt$$ 
on $\clh_{\phi}$ where $T_tx\zeta_{\phi}=\tau_t(x)\zeta_{\phi}$ is the strongly continuous 
semi-group of contraction on $\clh_{\phi}$. As $\lambda \raro 0$, $s_{\lambda} \raro P$ in strong operator topology follows along the same line where $P$ is the projection on the closed subspace $\{f: T_tf=f: t \ge 0 \}$. 
Thus we can use strong convergence $s'_{\lambda}(T') \raro P'$ as $\lambda \raro 0$, where $P'$ is the projection on the closed subspace
$\{f: T'_tf=f: t \ge 0 \}$ where $T'_ty'\zeta_{\phi}=\tau'_t(y')\zeta_{\phi},\;y' \in \clm'$. Once again same argument shows that $P'=P$ to conclude that 
$||\psi_{y'_1,y'_2}(s_{\lambda}-s_{\mu})|| \raro 0$ 
in $\clm_*$ as $\lambda,\mu \raro 0$. Once more using sequential weak$^*$ 
compactness of $\clm$ and duality argument used in Theorem 1.1, we conclude that 
$||\psi \circ s_{\lambda} - \psi \circ E||$ as $\lambda \raro 0$ 
for all $\psi \in \clm_*$. 
\end{proof}           

\vsp 
\begin{rem} 
Let $\tau_n:\clm \raro \clm$ be a sequence of automorphisms with an invariant normal faithful state $\phi$ in Theorem 1.1. Since 
modular automorphism group of $\phi$ commutes with an automorphism preserving $\phi$, dual maps $\tau'_n$ are also automorphisms. 
Thus by theorem 1.1 for all $\psi \in \clm_*$ we have $$||\psi \tau_n -\psi \tau|| \raro 0$$ for some automorphism $\tau$ on $\clm$ 
if and only if unitary operators $u_n \raro u$ in strong operator topology where $u_n x \zeta_{\phi} = \tau_n^{-1}(x)\zeta_{\phi}$ 
and $ux\zeta_{\phi} = \tau^{-1}(x)\zeta_{\phi}$. Since $u_n \raro u$ in strong operator topology if and only if $u_n^* \raro u^*$ in 
strong operator topology, we can as replace the criteria with $\tau_n$ instead it's inverse (all elements being unitary 
$u_n \raro^{strong-op} u$ is equivalent to $u_n \raro^{weak-op} u$). Further we now specialize $\clm=L^{\infty}(\Omega,\clf,\mu)$ i.e. is a commutative von-Neumann 
algebra and $\tau_n(f)= f \circ \gamma_n$ for some probability measure $\mu$ preserving one to one and onto map bi-measurable maps $\gamma_n:
\Omega \raro \Omega$ modulo $\mu$ null set. Theorem 1.1 says that for some $\mu$ preserving bi-measurable one -one and onto map $\gamma:\Omega 
\raro \Omega$ we have 
$$||\psi \circ \gamma_n - \psi \circ \gamma||_1 \raro 0$$ 
for all $\psi \in L^1(\Omega,\clf,\mu)$ if and only if 
$$||f \circ \gamma_n - f \circ \gamma||_2 \raro 0$$ 
for all $f \in L^2(\Omega,\clf,d\mu)$.     
\end{rem} 

\underline{•}
  
\section{Birkhoff's ergodic theorem for amenable group:} 

\vsp 
A locally compact group $G$ is called amenable if for any compact $K \subset G$ and $\delta > 0$ there 
exists a compact subset $F \subset G$ so that 
$|F \Delta KF| < \delta |F|$ 
where we use $|.|$ to denote the left Haar measure on $G$. Such a set $F$ is called 
$(K,\delta)-$invariant. An infinite sequence $F_1,F_2,...$ of compact subsets of $G$ will 
be called {\it F\o lner sequence} if for every compact $K$ and $\delta >0$, $F_n$ is 
$(K,\delta)$-invariant for all $n \ge N(K,\delta)$, where $N(K,\delta)$ is an integer 
depending on $K$ and $\delta.$    
 
\vsp 
Let $G$ be a locally compact second countable amenable group acting from left bi-measurably on a probability 
space $(X,\clf,\mu)$ such that 
$$\mu \circ g^{-1}(E)  = \mu(E)$$ 
for all $g \in G$ and $E \in \clf$. For an element $f \in L^1(x,\clb,\mu)$, we set $s(F,f)$ 
to denote the average 
$$s(F,f)(x)= {1 \over |F|} \int_F f(gx)dm(g)$$
where $dm$ is a left-invariant Haar-measure on $G$. 

\vsp 
A basic problem in ergodic theory is to make a wise choice for a F\o lner sequence so that 
$$s(F_n,f) \raro \bar{f}(x)\;\;\mbox{a.e.}$$
and also in norm topology of $L^1(X,\clb,\mu)$ as $n \raro \infty$ where 
$\bar{f}(x)=E(f|\clf_e)(x)$ is the conditional 
expectation $f$ on the $\sigma$-field of $G$-invariant sets i.e. 
$$\clf_e = \{ E: \mu(E \Delta g^{-1}(E))=0,\;\forall g \in G \}$$      

\vsp 
A. Shulman [Sh] introduced in his thesis such an useful concept to prove an $L^2$-version of ergodic theorem [Sh] 
which played a key role in Lindenstrauss $L^1$ version of ergodic theorem [Lin]. We briefly recall [Sh] that a F\o lner 
sequence $\{F_n:n \ge 1\}$ of compact subsets of $G$ is called tempered if there exists a constant $C > 0$ such that 
$$|\bigcup_{ k < n}F_k^{-1}F_n| < C|F_n|$$ 
for all $n \ge 1$. Lindenstrauss [Lin] also proved that there exists a sub-sequence of a F\o lner sequence which satisfies 
A. Shulman's tempered condition. 

\vsp 
A. Shulman [Sh] proved that for each $f \in L^2(X,\clb,\mu)$  
$$s(F_n,f) \raro \bar{f}$$ in $L^2$ whenever $F_n$ is a tempered F\o lner sequence. In the following we put his 
result in an abstract set up.  

\vsp 
\begin{thm} 
[Sh] Let $G$ be a second countable locally compact amenable group and $F_n$ be a tempered F\o lner sequence for $G$. 
Let $g \raro u_g$ be a unitary representation on a Hilbert space $\clh$ so that the map $g \raro <f,u_gh>$ be continuous for any two elements $f,h \in \clh$. Then   
\be 
s(F_n,f) \raro Pf
\ee
in strong operator topology as $n \raro \infty$ where $P$ is the projection on the closed subspace 
$\{f: u_gf=f,\; \forall g \in G \}$ and  
$$s(F,f) = {1 \over |F|} \int_F u_gfdm(g)$$
\end{thm}
\vsp 
\begin{proof} 
It goes along the same line [Sh] once we notice his proof does not require to use explicit form 
of the group action on the Hilbert space $L^2(X,\clb,\mu)$. We skip the details. 
\end{proof}  

\vsp 
\begin{thm} Let $G,F_n$ be as in Theorem 3.1. Let $g \raro \alpha_g$ be a group of $*$-automorphism on $\clm$ 
such that $g \raro \psi(\alpha_g(x))$ is a measurable function for each $x \in \clm$ and $\psi \in \clm_*$ and 
$\phi$ be a faithful normal invariant state for $\alpha_g$. Then for any $\psi \in \clm_*$
$$||\psi \circ s(F_n) - \psi \circ E || \raro 0$$
as $n \raro \infty$ where 
$$s_n(x) = {1 \over |F_n|} \int_{F_n}\alpha_g(x)dm(g)$$
is a sequence of unital completely positive map and 
$E$ is the normal norm one projection from $\clm$ to 
$$\cln=\{x \in \clm: \alpha_g(x)=x \}$$     
\end{thm}

\vsp 
\begin{proof} We set unitary representation $g \raro u_g:x\zeta_{\phi}=\alpha_g(x)\zeta_{\phi},\;x \in \clm$ and 
set as before $s(F_n,f) = {1 \over |F_n|} \int_{F_n} u_gfdm(g)$ for $f \in [\clm\zeta_{\phi}]$. By Theorem 3.1 we have $s(F_n,f) \raro Pf$ in strongly for all $f \in \clh_{\phi}$
where $P$ is the projection on invariant vectors $\{f \in \clh_{\phi}:u_gf=f:\;g \in G \}$. Further $s_n(x) \raro E(x)$ in weak$^*$ topology by weak$^*$ compactness of the unit ball to some element 
$E(x)$ and faithfulness of the normal state $\phi$. So in particular we have $E(x)\zeta_{\phi} \in P(\clh_{\phi})$ and thus $E(x)\zeta_{\phi}=\alpha_g(E(x))\zeta_{\phi}$ for all $g \in G$. This shows that 
$E(x) \in \cln$ by separating property of $\zeta_{\phi}$ where 
$\cln=\{x:\alpha_g(x)=x;g \in G \}$. 

\vsp 
We claim that $P$ is the close span of the vectors in $\{x\zeta_{\phi}: x \in \cln \}$. One way inclusion is trivial. For the reverse direction, let $f$ be an element such that $u_gf=f$ for 
all $g \in G$ and orthogonal to $\{x\zeta_{\phi}:x \in \cln \}$. Then $<f,E(x)\zeta_{\phi}>=0$ for all $x \in \clm$. However we also have $<f,E(x)\zeta_{\phi}>=\mbox{lim}_{n \raro \infty}<f,s_n(x)\zeta_{\phi}>=\mbox{lim}_{n \raro \infty}< s(F^{-1}_n,f),x\zeta_{\phi}> = <f,x\zeta_{\phi}>$ since $u_gf=f$ for all $g \in G$ where $F^{-1}=\{g \in G: g^{-1} \in F \}$. 
$\zeta_{\phi}$ being cyclic for $\clm$, we conclude that $f=0$. This shows claimed equality of subspaces.  

\vsp 
Now by Theorem 1.1 applied to sequence of maps $x \raro s_n(x)$, which is a completely positive unital map on $\clm$, to conclude 
$$||\psi_{y'_1,y'_2} \circ s_n - \psi_{y'_1,y'_2} \circ s_m || \le ||(s'_n-s'_m)((y'_1)^*y'_2)\zeta_{\phi}|| 
\raro 0$$ as $m,n \raro \infty$ as $s'(F_n,f) \raro Pf$ in strongly by Theorem 3.1 where 
$$s'(F_n,f)= {1 \over |F_n|} \int u_{g^{-1}}fdm(g)$$ for $f \in \clh_{\phi}$. Thus the result follows along the same line of Theorem 1.1 since the set 
$\{ \psi_{y'_1,y'_2}:y_1',y_2' \in \clm' \}$ is total in $\clm_*$ in the Banach space topology. 
\end{proof}

\bigskip
{\centerline {\bf REFERENCES}}

\begin{itemize}

\bigskip 
\item {[AC]} Accardi, L., Cecchini, C.: Conditional expectations in von Neumann algebras and a theorem of Takesaki, Journal of Functional Analysis 45 (1982), 245-273, 

\item{[BR]} Bratteli, O., Robinson, D.W. : Operator algebras and quantum statistical mechanics, I Springer 1981.

\item {[CE]} Choi, M.D., Effros, E.: Injectivity and operator spaces, J. Func. Anal. 24, 156–209. (1974)

\item{[Di]} Dixmier, J.: von Neumann algebras, North Holland 1981. 

\item{[De]} Dell'Antonio, G. F.: On the limits of sequences of normal states. Comm. Pure Appl. Math. 20 (1967), pp 413-429. 

\item{[Fr]} Frigerio, A.: Stationary states of quantum dynamical semigroups. Comm. Math. Phys. 63 (1978), no. 3, 269-276. 

\item{[JXu]} Junge, M. Xu, Q.: Non-commutative maximal ergodic theorems. J. Amer. Math. Soc. 20 (2007), no. 2, 385-439. 

\item {[Ka]} Kadison, Richard V.: A generalized Schwarz inequality and algebraic invariants for operator algebras,  Ann. of Math. (2)  56, 494-503 (1952). 

\item{[Ku]} Kümmerer, B.: A non-commutative individual ergodic theorem. Invent. Math.  46  (1978), no. 2, 139-145.

\item{[Kr]} Krengel, U.: Ergodic Theorems, De Gruyter studies in Mathematics 6, Berlin. New York 1985.

\item{[La]} Lance, E.C.: Ergodic theorems for convex sets and operator algebras. Invent. Math. 37 (1976), no. 3, 201-214.

\item{[Lin]} Lindenstrauss, E.: Point-wise theorems for amenable groups. Invent. Math. 146 (2001), no. 2, pp 259-295. 

\item{[Ha]} Halmos, P.R.: Lecture on ergodic theory, Chelsea Pub. Co. New York 1956.

\item{[Mo1]} Mohari, A.: Markov shift in non-commutative probability, Jour. Func. Anal. 199 (2003) 189-209.  

\item{[Mo2]} Mohari, A.: Pure inductive limit state and Kolmogorov's property, J. Funct. Anal. vol 253, no-2, 584-604 (2007)
Elsevier Sciences.

\item{[Mo3]} Mohari, A.: Pure inductive limit state and Kolmogorov's property -II, http://arxiv.org/abs/1101.5961. To appear in Journal of Operator Theory. 

\item{[Ne]} Nelson, E. The adjoint Markoff process. Duke Math. J. 25 (1958) 671-690. 

\item {[OP]} Ohya, M., Petz, D.: Quantum entropy and its use, Text and monograph in physics, Springer 1995. 

\item{[Sh]} Shulman, A.: Maximal ergodic theorem in groups, Dep. Lit. NIINTI, No.2184 (1988).  

\item{[St]} Stinespring, W. F.: Positive functions on $C^*$-algebras.  Proc. Amer. Math. Soc.  6,  (1955)  
211-216. 

\item{[Tak]} Takesaki, M.: Theory of operator algebra -I, Springer 2001.  

\item{[Um]} Umegaki, H.: Conditional expectation in an operator algebra. Tôhoku Math. J. (2) 6, (1954) 
177-181.

\end{itemize}

\end{document}